\documentclass[10pt,reqno]{amsart}
\usepackage{bbm}
\usepackage{mathrsfs}
\usepackage{amsfonts} 
\usepackage[dvipsnames,usenames]{color}
\usepackage{graphicx,latexsym,bm,amsmath,amssymb,verbatim,multicol,lscape}
\usepackage{pgfplots}

\newtheorem{thm}{Theorem} [section]

\newtheorem{lem}[thm]{Lemma}

\theoremstyle{definition}
\newtheorem{defn}[thm]{Definition}
\theoremstyle{remark}
\newtheorem{rem}[thm]{Remark}
\numberwithin{equation}{section}

\newcommand{\abs}[1]{\left\vert#1\right\vert}
\newcommand{\set}[1]{\left\{#1\right\}}

\newcommand{\Z}{\mathbb{Z}}
\newcommand{\F}{\mathbb{F}}

\newcommand{\Q}{\mathbb{Q}}

\newcommand{\E}{\mathcal{E}}
\newcommand{\N}{\mathcal{N}}

\newcommand{\OO} {\mathcal O}
\newcommand{\p} {\mathfrak p}
\newcommand{\ha} {\mathfrak a}
\newcommand{\hb} {\mathfrak b}
\newcommand{\ang}[1]{\left\langle#1\right\rangle}
\begin{document}
\title{Algebraic properties of summation of exponential Taylor polynomials}
\begin{abstract}
Let $n\ge 1$ be an integer and $e_n(x)$ denote the truncated
exponential Taylor polynomial, i.e.  $e_{n}(x)=\sum_{i=0}^n\frac{x^i}{i!}$.
A well-known theorem of Schur states that the Galois group of
$e_n(x)$ over $\Q$ is the alternating group $A_n$ if $n$ is
divisible by 4 or the symmetric group $S_n$ otherwise.
In this paper, we study algebraic properties of the summation of
two truncated exponential Taylor polynomials $\E_n(x):=e_n(x)+e_{n-1}(x)$.
We show that $\frac{x^n}{n!}+\sum_{i=0}^{n-1}c_i\frac{x^i}{i!}$
with all $c_i \ (0\le i\le n-1)$ being integers
is irreducible over $\Q$ if either $c_0=\pm 1$, or $n$ is not a positive
power of $2$ but $|c_0|$ is a positive power of 2.
This extends another theorem of Schur. We show also that $\E_n(x)$
is irreducible if $n\not\in\{2,4\}$. Furthermore, we show that
${\rm Gal}_{\Q}(\E_n)$ contains $A_{n}$ except
for $n=4$, in which case, ${\rm Gal}_{\Q}(\E_4)=S_3$. Finally, we show that
the Galois group ${\rm Gal}_{\Q}(\E_n)$ is $S_n$ if $n\equiv 3 \pmod 4$, or
if $n$ is even and $v_p(n!)$ is odd for a prime divisor of $n-1$, or if
$n\equiv 1\pmod 4$ and $n-2$ equals the product of an odd prime number
$p$ which is coprime to $\sum_{i=1}^{p-1}2^{p-1-i}i!$ and a positive integer
coprime to $p$.
\end{abstract}

\author[L.F. Ao]{Lingfeng Ao}
\address{Mathematical College, Sichuan University, Chengdu 610064, P.R. China}
\email{alf0038@126.com}
\author[S.F. Hong]{Shaofang Hong$^*$}
\address{Mathematical College, Sichuan University, Chengdu 610064, P.R. China}
\email{sfhong@scu.edu.cn; s-f.hong@tom.com; hongsf02@yahoo.com}
\thanks{$^*$S.F. Hong is the corresponding author and was supported
partially by National Science Foundation of China Grant \#11771304.}
\keywords{Newton polygon, irreduciblity Criterion, Galois group,
symmetric group, alternating group.}
\subjclass[2000]{Primary 11R09, 11R32, 11C08}
\maketitle
\section{Introduction}
Let $n\ge 1$ be an integer and $F$ be an arbitrary field. Let $f(x)\in F[x]$ be a polynomial of degree $n$. We denote by ${\rm spl}_{F}f$ the splitting field of $f(x)$ over $F$. The Galois group of $f(x)$ over $F$ always means the Galois group of the field extension ${\rm spl}_{F}f/F$, which is denoted by ${\rm Gal}_{F}(f)$.
Let $e_n(x)$ stand for the $n$-th truncated exponential Taylor polynomials, i.e. $e_n(x):=\sum_{i=0}^{n}{x^i}/{i!}$. These polynomials are irreducible over the field $\Q$ of rational numbers, which was first proved by Schur\cite{Schur1}. Actually, he proved a more general result stating that any polynomial of the form
$$1+a_1x+a_2\frac{x^2}{2!}+\cdots+a_{n-1}\frac{x^{n-1}}{(n-1)!}\pm\frac{x^n}{n!}$$
are irreducible over $\Q$, where $a_i\in \Z$ for all positive integers $i$ with
$\Z$ being the set of all integers.  Coleman\cite{Coleman} reproved the irreducibility
of $e_n(x)$ by the $p$-adic Newton polygon method. Coleman and Schur showed that  ${\rm Gal}_{\Q}(e_n)$ is $S_n$ (full symmetric group) if $n\not \equiv 0 \pmod 4$ or $A_n$ (alternating group) otherwise.

The truncated exponential Taylor polynomials are special cases of
generalized Laguerre polynomials (GLP) which is a one-parameter family defined by
$$L_n^{(\alpha)}(x)=(-1)^n\sum\limits_{j=0}^{n}\binom{n+\alpha}{n-j}\frac{(-x)^j}{j!}.$$
There are some authors who studied the algebraic properties for certain rational values of $\alpha$. Schur \cite{Schur2} claimed the irreducibility of $L_n^{(\alpha)}(x)$ over $\Q$ and computed ${\rm Gal}_{\Q}(L_n^{(\alpha)})$ for $\alpha=0,1$ and $-n-1$. Hajir \cite{Hajir1} did the same for $\alpha=-n-2$.
Notice that $L_n^{(-n-1)}(x)=e_{n}(x)$ and $L_n^{(-n-2)}(x)=\sum_{j=0}^{n}e_{j}(x)$. For an account of results on GLP, we refer the readers to Hajir \cite{Hajir3}, Sell \cite{Sell} and Filaseta, Kidd and Trifonov \cite{Filaseta}.

It is natural to ask what is the Galois group for other kind of polynomials. In this paper, we mainly consider algebraic properties of the family of summation of exponential Taylor polynomials. Let $\E_{n}(x)$ denote the summation of $e_{n}(x)$ and $e_{n-1}(x)$, i.e.
$$\E_{n}(x):=e_n(x)+e_{n-1}(x)=\frac{x^n}{n!}+2\sum_{i=0}^{n-1}\frac{x^i}{i!}.$$
We are interested in the irreducibility of $\E_n(x)$ and the Galois group ${\rm Gal}_{\Q}(\E_n)$. First of all, we prove a generalization of Schur's
theorem.

\begin{thm}\label{thm1}
Let
\begin{align}\label{irr}
f(x)=\sum\limits_{i=0}^{n}c_i\frac{x^i}{i!}\in\Q[x]
\end{align}
be a polynomial of degree $n\ge 2$, where $c_i\in \Z$
for all integers $i$ with $1\le i\le n-1$, $c_n=1$
and $c_0=\pm 2^{k}$ with $k$ being a nonnegative integer.
Then each of the following is true:

{\rm (i).} Either $f(x)$ is irreducible over $\Q$, or
$f(x)=g(x)\prod_{i=1}^{n-\deg g(x)}(x\pm 2^{t_i})$ with
$g(x)$ being irreducible and $\frac{n}{2}<\deg g(x)<n$
and all $t_i \ (1\le i\le n-\deg g(x))$ being nonnegative integers.

{\rm (ii).} If $n$ is not a positive power of $2$
or $c_0=\pm 1$, then $f(x)$ is irreducible over $\Q$.
\end{thm}

Using Theorem \ref{thm1} and a result of Hajir \cite{Hajir1},
we can deduce the following result about the irreducibility
and ${\rm Gal}_{\Q}(\E_n)$.

\begin{thm}\label{thm2}
Let $n\ge 1$ be an integer. Then each of the following is true:

{\rm (i).} The polynomial $\E_{n}(x)$ is irreducible
over $\Q$ except for $n\in\{2, 4\}$.

{\rm (ii).} ${\rm Gal}_{\Q}(\E_n)$ contains $A_{n}$ except
for $n=4$, in which case, ${\rm Gal}_{\Q}(\E_4)=S_3$.
\end{thm}

To determine ${\rm Gal}_{\Q}(\E_n)$ is $A_n$ or $S_n$, one
needs to determine whether the discriminant of $n!\E_n(x)$
is a square of a rational number or not. To do this, we
obtain a formula for the discriminant of $n!\E_n(x)$ at first.
Then we use it to show that ${\rm Gal}_{\Q}(\E_n)$ is $S_n$ in some
special cases. Namely, one has the following result.

\begin{thm}\label{thm4}
The Galois group ${\rm Gal}_{\Q}(\E_n)$ equals $S_n$
in each of the following cases:

{\rm (i).} $n\equiv 3 \pmod 4$.

{\rm (ii).} $n$ is even, and $v_p(n!)$ is odd for a prime
divisor of $n-1$.

{\rm (iii).} $n\equiv 1\pmod 4$ and $n-2$ equals the product
of an odd prime number $p$ and a positive integer which is
coprime to $p$ with $p$ satisfying that
\begin{align}\label{cond}
\sum\limits_{i=1}^{p-1}2^{p-1-i}i!\not \equiv 0\pmod p.
\end{align}
\end{thm}

\begin{rem}
We point out that there are infinitely many positive integers
$n$ satisfying the condition presented in Theorem {\ref{thm4}} (iii)
for which one has ${\rm Gal}_{\Q}(\E_n)=S_n$. For instance,
taking $p=3$ gives that $\sum_{i=1}^{2}2^{3-1-i}i! \equiv 1\pmod3$
which satisfies the condition (\ref{cond}). By Chinese
reminder theorem, one finds that any positive integer
$n$ with $n\equiv 5\pmod {36}$ or $n\equiv 17 \pmod {36}$
satisfies the condition in part (iii) of Theorem \ref{thm4}.
Thus ${\rm Gal}_{\Q}(\E_n)=S_n$ if $n\equiv 5\pmod {36}$
or $n\equiv 17 \pmod {36}$.
\end{rem}

This paper is organised as follows. We give some definitions
and preliminary lemmas in the second section. Then we prove
Theorem {\ref{thm1}} in the third section. The fourth section
is devoted to the proof of Theorem {\ref{thm2}}.
In the last section, we present the proof of Theorem {\ref{thm4}}.
\section{Preliminary lemmas}
In this section, we present some definitions and preliminary
lemmas. As usual, we use $p$ to denote a prime number.
We denote by  $s_p(n)$ and $\left(\frac{\cdot}{p}\right)$
the digital summation of the $p$-adic expansion of $n$ and
the Legendre symbol with respect to $p$. Given a prime number
$p$, $\F_p$ denotes the finite field of $p$ elements, i.e. $\F_p=\Z/p\Z=\set{0,1,...,p-1}$. For a given polynomial
$h(x)\in \Z[x]$ and a prime number $p$, we use $\ang{h(x)}_p$
to denote $h(x)$ modulo $p$, i.e.
$\ang{h(x)}_p\equiv h(x)\pmod p$.

\begin{defn}
The {\it $p$-adic valuation} of an integer $m$ with respect to
$p$, denoted by $v_p(m)$, is defined as
$$
v_p(m)=\left\{
\begin{array}{ll}{\max\{k\ge 0: p^k\mid m\}}
& { ~\text{if}~~ m\ne 0,} \\ {\infty} & { ~\text{if}~~ m=0.}
\end{array}\right.
$$
\end{defn}
This definition could be naturally extended to the rational field
$\Q$, any algebraic extension of $\Q$ and the local field $\Q_p$.
To prove Theorem \ref{thm1}, we need the following lemma due to
Sylvester which is a generalization of Bertrand's postulate.

\begin{lem} {\rm \cite{Sylvester}} \label{Erdos}
For any positive integer $k\le l$, at least one of
the numbers in the list
$$l+1,l+2,...,l+k$$
is divisible by a prime number which is greater than $k$.
\end{lem}
\begin{proof}
This was proved by Sylvester \cite{Sylvester}, and reproved
by Schur \cite{Schur1} and Erdos \cite{Erdos}.
\end{proof}

We require Legendre theorem for calculating the
$p$-adic valuation for the factorial $n!$.

\begin{lem}\label{Legendre}\textup{\cite{Koblitz}}
For every positive integer $n$ and prime $p$, we have
$$v_p(n!)=\frac{n-s_p(n)}{p-1}.$$
\end{lem}

\begin{defn}
The {\it $p$-adic Newton polygon} $NP_{p}(f)$ of a polynomial $f(x)=\sum\limits_{j=0}^{n}c_jx^{j}\in\Q[x]$ is the lower convex hull of the set of points
$$S_p(f)=\{{(j,v_p(c_j))\mid 0\leq j \leq n}\}.$$
The vertices $(x_0,y_0),(x_1,y_1),\ldots,(x_r,y_r)$, the points where the slope of the Newton polygon changes are called the {\it corners} of $NP_{p}(f)$; their $x$-coordinates $(0=x_0<x_1<...<x_r=n)$ are the {\it breaks} of $NP_{p}(f)$; the lines connecting two vertices are called the {\it segments} of $NP_{p}(f)$.
\end{defn}

\noindent Evidently, the $p$-adic Newton polygon $NP_{p}(f)$ is the highest polygonal
line passing on or below the points in $S_p(f)$.

\begin{lem}\label{MTNP} \textup{\cite{Hajir3}} {\textup{(Main theorem on $p$-adic Newton polygon).}}
~Let $(x_0,y_0),(x_1,y_1),...,(x_r,y_r)$ denote the successive vertices of
$p$-adic Newton polygon $NP_p(f)$ of $f(x)$. Then there exist polynomials
$f_1,...,f_r$ in $\Q_{p}[x]$ such that each of the following is true:

{\rm (i).}{\it~$f(x)=f_1(x)f_2(x)\cdots f_r(x)$.}

{\rm (ii).}{\it The degree of $f_i$ is $x_i-x_{i-1}$.}

{\rm (iii).} {\it All the roots of $f_i$ in $\overline{\Q_{p}}$
have $p$-adic valuation $-\frac{y_i-y_{i-1}}{x_i-x_{i-1}}$.}
\end{lem}

\begin{lem}\label{emtwo}
For any even integer $n\ge 6$, we have $\E_n(-2)> 0$.
\end{lem}
\begin{proof}
By straight forward calculation, we have $\E_2(-2)=\E_4(-2)=0$.
In the following, we assume that $n\ge 6$.
We claim that $\E_n(-2)$ is an increasing function for even
$n$. Since $n$ is even, we have
\begin{align*}
\E_{n}(-2)-\E_{n-2}(-2)&=\frac{(-2)^n}{n!}
+2\sum\limits_{i=0}^{n-1}\frac{(-2)^i}{i!}
-\Big(\frac{(-2)^{n-2}}{{n-2}!}
+2\sum\limits_{i=0}^{n-3}\frac{(-2)^i}{i!}\Big)
\\&=\frac{(-2)^n}{n!}
+2\frac{(-2)^{n-1}}{(n-1)!}+\frac{(-2)^{n-2}}{(n-2)!}\\
&=\frac{2^{n-2}}{n!}\Big(2^2-n2^2+n(n-1)\Big)\\
&=\frac{2^{n-2}}{n!}(n-1)(n-4)>0.
\end{align*}
This finishes the proof of Lemma \ref{emtwo}.
\end{proof}

Lemmas \ref{MTNP} and \ref{emtwo} are the key ingredients in the proof of Theorem \ref{thm2} (i). Lemma \ref{MTNP} tells us that the $2$-adic valuation of each roots of $\E_n(x)$ equals to $1$ when $n=2^s$ for some positive integer $s$. Lemma \ref{emtwo} infers that $-2$ cannot be a root of $\E_n(x)$ if $s>3$.

To treat the case when $n$ is large, we need the definition of Newton index and a theorem of Hajir which plays an important role in the computation of the Galois groups for polynomials. With the help of a theorem of Chebyshev on the existence of a prime between $n/2$ and $n-2$ for all $n\ge 8$, one derives the truth of Theorem \ref{thm2} (ii) for all $n\ge 8$.

\begin{defn}
Given $f\in \Q[x]$, let $\N_f$ be the least common multiple
of the denominators (in lowest terms) of all slopes of the
$p$-adic Newton polygon $NP_p(f)$ as $p$ ranges over all primes.
Such $\N_f$ is called the {\it Newton index} of $f$.
\end{defn}

\begin{lem}\label{HAJIR}\textup{\cite{Hajir}}
For any irreducible polynomial $f\in \Q[x]$
of degree $n$, $\N_{f}$ divides the order of
${\rm Gal}_{\Q}(f)$. Moreover, if $\N_{f}$
has a prime divisor $q$ in the range
$\frac{n}{2}<q <n-2$, then ${\rm Gal}_{\Q}(f)$ contains $A_n$.
\end{lem}

\begin{lem}\label{CHEBY}\textup{\cite{Chebyshev}}
 For each integer $n\geq 8$, there exist a prime number $q$
strictly between $n/2$ and $n-2$.
\end{lem}

To handle the small values of $n$, we provide following
lemmas from standard Galois theory. Lemmas \ref{Subgroup}
and \ref{Symmetric} transfer the calculation of
${\rm Gal}_{\Q}(\E_n)$ into a concrete way.

\begin{lem}\label{Subgroup}\textup{\cite{Grillet}}
Let $f(x)\in \Z[x]$ with no repeated roots and the leading
coefficient $1$. Let $p$ be a prime number which is
coprime to the discriminant of $f(x)$ over $\Q$. Then ${\rm Gal}_{\F_p}(\ang{f}_p)$ is a subgroup of
${\rm Gal}_{\Q}(f)$.
\end{lem}

\begin{lem}\label{Subgroups}
Let $f(x)$ be a separable monic polynomial with integral coefficients of degree $n$.
Let $p$ be a prime number which is coprime to the discriminant of $f(x)$ over $\Q$. Then each
of the following is true:

{\rm (i).} If $\ang{f(x)}_p=\ang{x+a}_p\ang{g(x)}_p$ with $\ang{g(x)}_p$
irreducible over $\F_p$, then ${\rm Gal}_{\Q}(f)$ contains an $(n-1)$-cycle.

{\rm (ii).} If $n$ is odd and $\ang{f(x)}_p=\ang{h(x)}_p\ang{g(x)}_p$ with
$\deg \ang{h(x)}_p=2$ and both of $\ang{h(x)}_p$ and
$\ang{g(x)}_p$ are irreducible over $\F_p$, then
${\rm Gal}_{\Q}(f)$ contains a transposition.

{\rm (iii).} If $n$ is even and $\ang{f(x)}_p=\ang{(x+a)}_p\ang{h(x)}_p\ang{g(x)}_p$
with $\deg \ang{h(x)}_p=2$ and both of $\ang{h(x)}_p$ and $\ang{g(x)}_p$ are irreducible
over $\F_p$, then ${\rm Gal}_{\Q}(f)$ contains a transposition.
\end{lem}

\begin{proof}
{\rm (i).} Since $\ang{f(x)}_p=\ang{x+a}_p\ang{g(x)}_p$ and $\ang{g(x)}_p$
irreducible over $\F_p$, we have ${\rm spl}_{\F_p}(\ang{f}_p)={\rm spl}_{\F_p}(\ang{g}_p)$.
By the structure theorem of finite field, one has ${\rm spl}_{\F_p}(\ang{g}_p)=\F_{p^{n-1}}$.
So ${\rm Gal}_{\F_p}(\ang{f}_p)={\rm Gal}_{\F_p}(\ang{g}_p)={\rm Gal}(\F_{p^{n-1}}/\F_p)$ is a cyclic group of order $n-1$. By Lemma \ref{Subgroup}, ${\rm Gal}_{\F_p}(\ang{f}_p)$ is a subgroup of ${\rm Gal}_{\Q}(f)$.
Hence ${\rm Gal}_{\Q}(f)$ contains an $(n-1)$-cycle.

{\rm (ii).} Since $\ang{f(x)}_p=\ang{h(x)}_p\ang{g(x)}_p$ and $\deg \ang{h(x)}_p=2$,
we have $\deg \ang{g(x)}_p=n-2$. Then the structure theorem of finite field tells us that ${\rm spl}_{\F_p}(\ang{h}_p)=\F_{p^{2}}$ and ${\rm spl}_{\F_p}(\ang{g}_p)=\F_{p^{n-2}}$. Since $\ang{f(x)}_p=\ang{h(x)}_p\ang{g(x)}_p$, we have ${\rm spl}_{\F_p}(\ang{f}_p)$ is the
composition ${\rm spl}_{\F_p}(\ang{h}_p){\rm spl}_{\F_p}(\ang{g}_p)$ of ${\rm spl}_{\F_p}(\ang{h}_p)$ and ${\rm spl}_{\F_p}(\ang{g}_p)$.
Since $n$ is odd, we have $\gcd(2,n-2)=1$ which implies that ${\rm spl}_{\F_p}(\ang{f})=\F_{p^{2(n-2)}}$. It follows that ${\rm Gal}_{\F_p}(\ang{f})$ is
a cyclic group of order $2(n-2)$.

Let ${\rm Gal}_{\F_p}(\ang{f}_p)=\ang{\sigma}$ with $\sigma$ being the Frobenius action.
We prove that $\sigma^{n-2}$ is a transposition on the roots of $\ang{f(x)}_p$.
Let $\alpha_1$ and $\alpha_2$ be the roots of $h(x)$ and $\beta_1,...,\beta_{n-2}$
be the roots of $g(x)$. Since $\beta_i\in \F_{p^{n-2}}$ for each $i\in \set{1,2,...,n-2}$, we have $\sigma^{n-2}(\beta_i)=\beta_i^{p^{n-2}}=\beta_i$.
So $\sigma^{n-2}$ fixes the roots of $g(x)$. But $\alpha_1,\alpha_2\in \F_{p^2}$.
Thus $\sigma^{n-2}(\alpha_1)=\alpha_1^{p^{n-2}}=\alpha_1^p=\alpha_2$
and $\sigma^{n-2}(\alpha_2)=\alpha_2^{p^{n-2}}=\alpha_2^p=\alpha_1$.
Hence $\sigma^{n-2}$ is a transposition on the roots of $\ang{f(x)}_p$, i.e.
${\rm Gal}_{\F_p}(\ang{f}_p)$ contains a transposition. By Lemma \ref{Subgroup}, ${\rm Gal}_{\F_p}(\ang{f}_p)$ is a subgroup of ${\rm Gal}_{\Q}(f)$.
So ${\rm Gal}_{\Q}(f)$ contains a transposition.

{\rm (iii).} Let $\ang{f'(x)}_p=\ang{h(x)}_p\ang{g(x)}_p$. By Lemma \ref{Subgroups} (ii),
we have ${\rm Gal}_{\F_p}(\ang{f'}_p)$ contains a transposition. Since
${\rm Gal}_{\F_p}(\ang{f'}_p)={\rm Gal}_{\F_p}(\ang{f}_p)$, ${\rm Gal}_{\F_p}(\ang{f}_p)$
contains a transposition.
By Lemma \ref{Subgroup}, ${\rm Gal}_{\F_p}(\ang{f}_p)$ is a subgroup of ${\rm Gal}_{\Q}(f)$.
This infers that ${\rm Gal}_{\Q}(f)$ contains a transposition.
\end{proof}

\begin{lem}\label{Symmetric}\textup{\cite{Grillet}}
Let $G$ be a transitive permutation group of $n$ letters.
If $G$ contains a transposition and an $(n-1)$-cycle, then
$G=S_n$.
\end{lem}

To determine whether the Galois group of a polynomial is $A_n$ or $S_n$, we need the results on the discriminant of polynomials. At first, we present some facts about the discriminant in general.
Let $f(x)\in F[x]$ be a given monic polynomial of degree $n$.
Let $\alpha_1,...,\alpha_n$ be all the roots of $f(x)$. Then
$${\rm Disc}_{F}(f):=\prod_{1\le i<j \le n }(\alpha_i-\alpha_j)^{2}$$
is called the {\it discriminant} of $f(x)$ over $F$. It is well known
that
\begin{align}\label{DISC}
{\rm Disc}_{F}(f)=(-1)^{\frac{n(n-1)}{2}}\prod_{i=1}^n f'(\alpha_i),
\end{align}
 where $f'(x)$ is the usual formal derivative of $f(x)$.
The following result
computes the discriminant of $n!\E_{n}(x)$ over $\Q$.

\begin{lem}\label{DISCR}
The discriminant of $n!\E_{n}(x)$ over $\Q$ is equal to
$(-1)^{\frac{n(n-1)}{2}} 2^{n-1}(n!)^n\E_n(-n).$
\end{lem}
\begin{proof}
Since $\E_n(x)=e_n(x)+e_{n-1}(x)$, one has $$\E_n'(x)=e_n'(x)+e_{n-1}'(x)=e_{n-1}(x)+e_{n-2}(x).$$
It follows that
\begin{align}\label{eqee}
\E_n'(x)=\E_n(x)-\Big(\frac{x^n}{n!}+\frac{x^{n-1}}{(n-1)!}\Big).
\end{align}
Then by (\ref{DISC}), we have
\begin{align}\label{lem2.11}
{\rm Disc}_{\Q}(n!\E_n)=(-1)^{\frac{n(n-1)}{2}}\prod_{i=1}^{n}(n!\E_n'(\alpha_i)),
\end{align}
where $\alpha_1,...,\alpha_n$ are all the roots of $n!\E_n(x)$.
However, one has
$$\prod\limits_{i=1}^{n}\alpha_i=(-1)^n2n!$$
by Vieta theorem and $n!\E_n(-n)=\prod\limits_{i=1}^{n}(-n-\alpha_i)$.
Thus by (\ref{eqee}) and (\ref{lem2.11}) and noticing that
$\E_n(\alpha_i)=0$ for all $1\le i\le n$, we derive that
\begin{align*}
{\rm Disc}_{\Q}(n!\E_n)&=(-1)^{\frac{n(n-1)}{2}} \prod\limits_{i=1}^{n}\Big(n!\Big(\E_n(\alpha_i)-\Big(\frac{\alpha_i^n}{n!}+\frac{\alpha_i^{n-1}}{(n-1)!}\Big)\Big)\Big)\\
&=(-1)^{\frac{n(n-1)}{2}} \prod\limits_{i=1}^{n}(-\alpha_i^n-n\alpha_i^{n-1})\\
&=(-1)^{\frac{n(n-1)}{2}} \prod\limits_{i=1}^{n}\alpha_i^{n-1}\prod\limits_{i=1}^{n}(-\alpha_i-n)\\
&=(-1)^{\frac{n(n-1)}{2}} \Big(\prod\limits_{i=1}^{n}\alpha_i\Big)^{n-1}\Big(n!\E_n(-n)\Big)\\
&=(-1)^{\frac{n(n-1)}{2}} 2^{n-1}(n!)^n\E_n(-n)
\end{align*}
as required. This completes the proof of Lemma \ref{DISCR}.
\end{proof}

We also need the following standard result in Galois theory
that is about the relation between ${\rm Gal}_{F}(f)$ and the discriminant of $f(x)$ over $F$.

\begin{lem}\label{ANSN}\textup{\cite{Grillet}}
Let $f(x)\in F[x]$ be a polynomial of degree $n$. Then each of the following holds:

\textup{(i).} ${\rm Gal}_{F}(f)$ is transitive if and only if $f(x)$
is irreducible over $F$.

\textup{(ii).} If $\operatorname{char}(F) \neq 2$,
then ${\rm Gal}_{F}(f)\subseteq A_n$ if and only if ${\rm Disc}_{F}(f)$ is a square in $F$.
\end{lem}

\section{Proof of Theorem \ref{thm1}}

In this section, we provide the proof of Theorem {\ref{thm1}}.

{\it Proof of Theorem {\ref{thm1}}.}

{\rm (i)} Multiplying $n!$ to $f(x)$ gives us the polynomial $F(x)$ given as follows:
\begin{align}\label{eq2}
F(x)=n!f(x)=\sum\limits_{i=0}^{n}c_i\frac{n!}{i!}x^{i}
= x^n+\cdots+c_i\frac{n!}{i!}x^j
+\cdots\pm 2^{k}n!.
\end{align}
Assume that $f(x)$ is reducible over $\Q$. Then $F(x)$ is reducible
over $\Z$. Let $A(x)\in \Z[x]$ be an irreducible monic
factor of $F(x)$ with degree $m$ satisfying $1\le m\le n/2$.
Without loss of generality, we may write
$$A(x)=x^m+a_{m-1}x^{m-1}+\cdots+a_{1}x+a_0.$$

We {\sc claim} that if $a_0$ is divisible by an odd prime $p$,
then $p\le m$. To prove this claim, we let
$\alpha$ be a root of $A(x)$ and define $K=\Q(\alpha)$.
Then we have $[K:\Q]=m$. Since the coefficients of
$A(x)$ are integral, we have $\alpha\in \OO_K$.
It is known by standard algebraic number theory that
the principal ideal $(\alpha)$ in $\OO_K$ has norm
$\abs {N_{K/\Q}(\alpha)}=\abs{a_0}$.
One then derives that $p \mid N_{K/\Q}(\alpha)$
since $p|a_0$. Let $\p\subsetneq \OO_K$ be a prime ideal
lying above $p$. Since $p\mid a_0=\pm N_{K/\Q}(\alpha)$, we have
$$(\alpha)=\p^{d_{\p}}\ha, \textup{~and~} (p)=\p^{e_{\p}}\hb,$$
where $\gcd(\ha\hb,\p)=(1)$, the {\it $\p$-adic valuation}
is an extension of $v_p$ in $K$, and $v_{\p}(p)=e_{\p}$
and $v_{\p}(\alpha)=d_{\p}$. By Proposition 20 at page 24 in \cite{Lang}, we have $d_{\p}\ge1$ and $e_{\p}\ge1$.

Since $\alpha$ is a root of $A(x)$, we have $F(\alpha)=0$.
Then by (\ref{eq2}), we have
\begin{align*}
0= \alpha^{n}+nc_{n-1}\alpha^{n-1}+\cdots+c_j\frac{n!}{j!}\alpha^j+\cdots+n!c_{1}\alpha\pm2^kn!.
\end{align*}
So
\begin{align}\label{eq4}
\mp2^k n!=\alpha^n+nc_{n-1}\alpha^{n-1}+\cdots
+c_j\frac{n!}{j!}\alpha^j+\cdots+n!c_{1}\alpha.
\end{align}
Consider the $\p$-adic valuation of the left-hand
side of (\ref{eq4}). Since $p\ne 2$, we have $v_{\p}(2^k)=0$.
By Lemma \ref{Legendre}, we have
$$v_{\p}(\mp2^k n!)=v_{\p}(n!)=v_{\p}(p)v_{p}(n!)
=e_{\p}v_{p}(n!)=\frac{(n-s_p(n))e_{\p}}{p-1}.$$
Meanwhile, at least one term on the right-hand side of
(\ref{eq4}) has $\p$-adic valuation no more than
$v_{\p}(\mp2^k n!)$. Therefore we obtain that
\begin{align}\label{eq10}
\min_{1\le j\le n-1}\set{v_{\p}(\alpha^n), v_{\p}\Bigg(c_j\frac{n!}{j!}\alpha^j\Bigg)}\le v_{\p}(\mp2^k n!).
\end{align}
One has $v_{\p}(\pm \alpha^n)=nd_{\p}$ and
\begin{align}\label{eq11}
v_{\p}\Big(c_j\frac{n!}{j!}\alpha^j\Big)
=v_{\p}(c_j)+v_{\p}(n!)-v_{\p}(j!)+jd_{\p}.
\end{align}
Since $c_j\in \Z$, one has $v_{\p}(c_j)\ge 0$
for all $1\le j\le n-1$. Noticing that $v_{\p}(2^k)=0$,
it follows from (\ref{eq10}) and (\ref{eq11}) that for some
integer $j$ with $1\le j\le n$, we have
\begin{align}\label{eq5}
v_{\p}(n!)-v_{\p}(j!)+jd_{\p}\le v_{\p}\Big(c_j\frac{n!}{j!}\alpha^j\Big)
\le v_{\p}(\mp2^k n!)=v_{\p}(n!).
\end{align}
Then by Lemma \ref{Legendre} and (\ref{eq5}), we derive that
$$jd_{\p}\le v_{\p}(j!)=e_{\p} v_p(j!)
=\frac{(j-s_p(j))e_{\p}}{p-1}< \frac{je_{\p}}{p-1}.$$
Hence $(p-1)d_{\p}<e_{\p}\le m$ which implies that $p\le m$.
This finishes the proof of the claim.

Now we divide the proof of (i) into following two cases.

{\sc Case 1.} $m=1$. If $a_0$ is divisible
by an odd prime $p$, then the above claim infers
that $p\le m=1$, which is impossible.
Hence $a_0=\pm 2^t$ for $t\ge 0$. So $A(x)=x\pm 2^t$
as required.

{\sc Case 2.} $m\ge 2$. Since $2\le m\le n/2$,
we have $n\ge 4$. Thus $\frac{n!}{(n-m)!}$ is divisible
by some odd prime number. In what follows,
we show that each odd prime divisor
of $\frac{n!}{(n-m)!}$ divides $a_0$. Let $l$ be an odd
prime divisor of $\frac{n!}{(n-m)!}$. Since
$l\mid \frac{n!}{(n-m)!}$ and $c_i\in \Z$, it follows that
$c_i \frac{n!}{i!}\equiv 0 \pmod l$ for all integers
$i$ with $0\le i\le n-m$. We then derive that
$x^{n-m+1}\mid \ang{F(x)}_l$.
Now write $F(x)=A(x)B(x)$. Then
$x^{n-m+1}\mid \ang{A(x)}_l \cdot \ang{B(x)}_l$.
But $\ang{B(x)}_l$ is a polynomial with the
leading coefficient $\pm 1$ and $\deg \ang{B(x)}_l=n-m$.
Thus $x\mid\ang{A(x)}_l$ which is equivalent
to saying that $a_0\equiv 0\pmod l$, namely,
$l$ divides $A(0)=a_0$. Thus each odd prime divisor
of $\frac{n!}{(n-m)!}$ divides $a_0$. But the above claim
tells us that every prime divisor of $a_0$ is no more than $m$.
Hence every odd prime divisor of $\frac{n!}{(n-m)!}$ is no more
than $m$.

On the other hand, since $m\le \frac{n}{2}$, it follows that
$S=\set{n-m+1, ..., n-1, n}$ is a list of $m$ consecutive
integers which are all strictly greater than $m$.
So all the prime factors of each elements in $S$
are no more than $m$. But Lemma \ref{Erdos} applied
to the set $S$ gives us that at least one element in
$S$ is divisible by an odd prime strictly greater than $m$.
We arrive at a contradiction. Thus $f(x)$ contains no
an irreducible factor of degree between 2 and
$\frac{n}{2}$. In conclusion, we know that $f(x)$ is either irreducible over $\Q$, or
$f(x)=g(x)\prod_{i=1}^{n-\deg g(x)}(x\pm 2^{t_i})$ with
$g(x)$ being irreducible and $\frac{n}{2}<\deg g(x)<n$
and all $t_i \ (1\le i\le n-\deg g(x))$ being nonnegative integers.
This completes the proof of the first part of Theorem 1.1.

{\rm (ii)}
If $n\ne 2^s$ for any positive integer $s$, then either $n=1$
or $n$ contains an odd prime divisor. If $n=1$, then $F(x)$
is clearly irreducible. Now let $n$ be divisible by an
odd prime number $q$. Suppose that $F(x)$ is reducible over
$\Z$. Then by Theorem \ref{thm1}(i), $F(x)$ contains
a linear factor of the form $x\pm 2^t$. One may write
$F(x)=(x\pm 2^t)D(x)$, where
$D(x)=x^{n-1}+\sum\limits_{i=0}^{n-2}d_ix^{i}$
and $d_i\in \Z$ for $i=0,1,...,n-2$. For any integer $i$ with
$0\le i \le n-1$, since $q|n$ and $c_i\in \Z$, one has
$c_i\frac{n!}{i!}\equiv 0\pmod q$. Hence
$$\ang{x\pm 2^t}_q\cdot\ang{D(x)}_q
\equiv F(x)\equiv x^n \pmod q.$$
Therefore $x^n|\ang{x\pm 2^t}_q\cdot\ang{D(x)}_q$.
Since the degree of $D(x)$ is $n-1$, we deduce that
$x|\ang{x\pm 2^t}_q$ which is a contradiction
to $q$ is odd. Hence $F(x)$ must be irreducible in this case.

Now let $n=2^s$ and $c_0=\pm 1$. Suppose that $F(x)$
is reducible.
We consider the $2$-adic Newton polygon of $F(x)$.
For any integer $i$ with $1\le i\le n-1$, we have
\begin{align*}
v_2\Big(c_i\frac{2^s!}{i!}\Big)&=v_2(c_i)+v_2(2^s!)-v_2(i!)\\
&=v_2(c_i)+2^s-s_2(2^s)-i+s_2(i)\\
&=v_2(c_i)+2^s-i-1+s_2(i)\\
&\ge 0+2^s-i-1+1=2^s-i.
\end{align*}
The line connecting the two points $(0,2^s-1)$ and $(2^s,0)$
has equation $\frac{x}{2^s}+\frac{y}{2^s-1}=1$ in $xOy$-plane.
Hence for every integer $i$ with $1\le i\le n-1$, the point with
$i$ as its $x$-coordinate holds $y$-coordinate $y_0$ as follows:
$$y_0=(2^s-1)\Big(1-\frac{i}{2^s}\Big)=2^s-i-1+\frac{i}{2^s}
<2^s-i=v_2\Big(c_i\frac{2^s!}{i!}\Big).$$
Thus the point $(i,v_2(c_i\frac{n!}{i!}))$
lies over the segment connecting $(0,2^s-1)$ and $(2^s,0)$.
Therefore the $2$-adic Newton polygon of $F(x)$ is given
by the segment connecting the two vertices $(0,2^s-1),(2^s,0)$.
One immediately derives that the slope of the $2$-adic
Newton polygon of the segment corresponding to the
vertices $(0,2^s-1),(2^s,0)$ is $\frac{1-2^s}{2^s}$
which is in the interval $(-1, 0)$. It infers that the
2-adic valuation of all the zeros of $F(x)$ is in the
interval $(0, 1)$. But $\pm 2^t$ is a zero of $F(x)$.
Hence we have $0<v_2(\pm 2^t)<1$. This is impossible
since $v_2(\pm 2^t)=t$ is a nonnegative integer.
So $F(x)$ is irreducible. The second part of
Theorem \ref{thm1} is proved.

This concludes the proof of Theorem \ref{thm1}.
\hfill$\Box$

\section{Proof of Theorem \ref{thm2}}
In this section, we give the proof of Theorem \ref{thm2}.

{\it Proof of Theorem \ref{thm2}.}

{{\rm (i).}
Since $\E_2(x)=\frac{1}{2}(x+2)^2$ and $\E_4(x)=\frac{1}{24}(x+2)
(x^3+6x^2+12x+24)$, $\E_2(x)$ and $\E_4(x)$ are reducible as one desires.

In the following, we let $n\ne 2,4$. By Theorem \ref{thm1} (ii),
$\E_n(x)$ is irreducible over $\Q$ except when $n$ is a positive
power of 2. So we may let $n=2^s\ge 8$ in what follows. Since
$$v_2\Big(\frac{2}{i!}\Big)=1-i+s_2(i), v_2\Big(\frac{1}{n!}\Big)=2^s-1,$$
by the same argument in the proof of Theorem \ref{thm1} (ii),
the points $\Big(i,v_2\Big(\frac{2}{i!}\Big)\Big)$ lies above the segment
connecting $(0,1),(2^s,1-2^s)$ for each integer $i$ with $1\le i\le n-1$ .
Hence the $2$-adic Newton polygon of $\E_n(x)$ has vertices as follows:
$(0,1),(2^s,1-2^s)$. By Lemma \ref{MTNP} (iii), all roots of
$\E_n(x)$ have 2-adic valuation $-\frac{0-2^s}{1-(1-2^s)}=1$.
Therefore the possible roots of $\E_n(x)$ in $\Q$ are $2$ and $-2$.
One can easily check that $\E_n(2)>0$. By Lemma \ref{emtwo},
$\E_n(-2)>0$ if $n\ge 6$. Hence $2$ and $-2$ are not roots of
$\E_n(x)$ if $n\ge 6$. Thus part (i) is proved.

{\rm (ii).}
If $n\ge 8$, then Lemma \ref{CHEBY} guarantees the existence
of a prime number $q$ with $\frac{n}{2}<q<n-2$. It follows that
the $q$-adic Newton polygon of $\E_n(x)$ has the following three
vertices:
$$(0,0),(q,-1),(n,-1).$$
By the definition of $N_{\E_n}$, one has $q|\N_{\E_n}$.
By Theorem \ref{thm2} (i) and Lemma \ref{HAJIR},
${\rm Gal}_{\Q}(\E_n)\supseteq A_n$ for all integers $n\ge 8$.

Now we deal with the remaining case $1\le n\le 7$. Clearly,
$A_n\subseteq {\rm Gal}_{\Q}(\E_n)$ if $n\in \{1, 2\}$.
Therefore we only need to consider the case $3\le n\le 7$
that will be done case by case in what follows.

{\sc Case 1:} $n=3$. By Eisenstein irreducibility
criterion, $\E_3(x)=\frac{1}{6}(x^3+6x^2+12x+12)$
is irreducible over $\Q$. It is well-known that the Galois group of
an irreducible polynomial of integer coefficients
of degree $3$ is either $A_n$ or $S_n$.
Hence one has $A_3\subseteq {\rm Gal}_{\Q}(\E_3)$.

{\sc Case 2:} $n=4$. Since $$\E_4(x)=\frac{1}{24}(x+2)(x^3+6x^2+12x+24)=\frac{1}{24}(x+2)\E'_4(x),$$
we have ${\rm Gal}_{\Q}(\E_4)={\rm Gal}_{\Q}(\E'_4)$. By Eisenstein irreducibility
criterion, $\E'_4(x)$ is irreducible over $\Q$. So $A_3\subseteq {\rm Gal}_{\Q}(\E'_4)$.
We can compute and get that ${\rm Disc}_{\Q}(\E'_4)=-2^83^3$ which is not square
of a rational integer. Then by Lemma \ref{ANSN}, $A_3\ne{\rm Gal}_{\Q}(\E'_4)$.
It follows from the well-known fact that ${\rm Gal}_{\Q}(\E'_4)\subseteq S_3$,
so we have ${\rm Gal}_{\Q}(\E'_4)=S_3$. Thus Theorem 1.2 is true in this case.

Since ${\rm Gal}_{\Q}(\E_n)={\rm Gal}_{\Q}(n!\E_n)$, from now on we
need only to consider the monic polynomial $n!\E_n(x)$
for the remaining case $n\in \{5,6,7\}$, in which case,
we assert that ${\rm Gal}_{\Q}(n!\E_n)=S_n$.

{\sc Case 3:} $n=5$. Then
$$5!\E_5(x)=x^5+10x^4+40x^3+120x^2+240x+240.$$
By Lemma \ref{DISCR}, one gets that ${\rm Disc}_{\Q}(5!\E_5)=2^{16}3^55^511$.
Considering $5!\E_5(x)$ modulo $17$, one has
$$5!\E_5(x)\equiv f_1(x)f_2(x)\pmod {17},$$
where $f_1(x)=x^2+9x+4$ and $f_2(x)=x^3+x^2+10x+9$ are both
irreducible over the prime field $\F_{17}$. By Lemma \ref{Subgroups}(ii),
${\rm Gal}_{\Q}(5!\E_5)$ contains a transposition.

Now considering $5!\E_5$ modulo $61$, we have
$$5!\E_5(x)\equiv (x+14)(x^4+57x^3+35x^2+57x+52)\pmod {61}.$$
By Lemma \ref{Subgroups}(i), ${\rm Gal}_{\Q}(5!\E_5)$ contains a $4$-cycle. By Lemma \ref{Symmetric}, one derives that
${\rm Gal}_{\Q}(5!\E_5)=S_5$.

{\sc Case 4:} $n=6$. We proceed in the same way as
in Case 3, but we need to choose different primes.
First, we have
$$6!\E_6(x)=x^6+12x^5+60x^4+240x^3+720x^2+1440x+1440.$$
A direct computation using Lemma \ref{DISCR} gives us that
${\rm Disc}_{\Q}(6!\E_6)=2^{30}3^{12}5^{5}7$. Consider $6!\E_6(x)$
modulo $47$, one has
$$6!\E_6(x)\equiv (x+45)f_3(x)f_4(x)\pmod {47},$$
where $f_3(x)=x^2+2x+17$ and $f_4(x)=x^3+12x^2+24$ are both
irreducible over the prime field $\F_{47}$.
By Lemma \ref{Subgroups}(iii), ${\rm Gal}_{\Q}(6!\E_6)$ contains a transposition.

Let us now consider $6!\E_6(x)$ modulo $13$, one has
$$6!\E_6(x)\equiv (x+12)(x^5+8x^3+x^2+6x+3)\pmod {13}.$$
By Lemma \ref{Subgroups}(i), ${\rm Gal}_{\Q}(6!\E_6)$ contains a $5$-cycle.
It follows from Lemma \ref{Symmetric} that ${\rm Gal}_{\Q}(6!\E_6)=S_6$.

{\sc Case 5:} $n=7$. Then
$$7!\E_7(x)=x^7+14x^6+84x^5+420x^4+1680x^3+5040x^2+10080x+10080,$$
and ${\rm Disc}_{\Q}(7!\E_7)=-2^{30}3^{12}5^{7}7^7\times 11\times 79$.
First we consider $7!\E_7(x)$ modulo 13. We have
$$7!\E_7(x)\equiv f_5(x)f_6(x)\pmod {13},$$
where $f_5(x)=x^2+5x+10$ and $f_6(x)=x^5+9x^4+3x^3+3x^2+10x+7$ are both
irreducible over the prime field $\F_{13}$.

By Lemma \ref{Subgroups}(ii), we obtain that ${\rm Gal}_{\Q}(7!\E_7)$ contains a transposition.
Now considering $7!\E_7(x)$ modulo $61$, we have
$$7!\E_7(x)\equiv(x+50)(x^6+25x^5+54x^4+38x^3+24x^2+58x+43)\pmod {61}.$$
Hence by Lemma \ref{Subgroups}(i), ${\rm Gal}_{\Q}(7!\E_7)$ contains a $6$-cycle.
It follows from Lemma \ref{Symmetric} that ${\rm Gal}_{\Q}(7!\E_7)=S_7$.

This completes the proof of Theorem \ref{thm2}.
\hfill$\Box$
}
\section{Proof of Theorem \ref{thm4}}
In this section, we prove Theorem \ref{thm4}. For this purpose,
for any rational numbers $a$ and $b$, we write $a\sim b$ if
$a=bc^2$ for some rational number $c>0$.

{\it Proof of Theorem \ref{thm4}.}
It is a well-known fact that for any polynomial $f(x)$ of degree $n$,
${\rm Gal}_{\Q}(f)\subset S_n$. By Theorem \ref{thm2}, we have $A_n\subset{\rm Gal}_{\Q}(n!\E_n)\subset S_n$ for all $n\ne 4$. By Lemma \ref{ANSN} (ii),
we have ${\rm Gal}_{\Q}(n!\E_n)\subset A_n$ holds if and only if ${\rm Disc}_{\Q}(n!\E_n)$
is a square of a rational number. So in order to show that ${\rm Gal}_{\Q}(n!\E_n)=S_n$
in parts (i) to (iii), it is sufficient to show that ${\rm Disc}_{\Q}(n!\E_n)$ is
not a square of a rational number.
By Lemma \ref{DISCR}, one has
$${\rm Disc}_{\Q}(n!\E_n)=(-1)^{\frac{n(n-1)}{2}} 2^{n-1}(n!)^{n-1}n!\E_n(-n).$$

(i). $n\equiv 3 \pmod 4$. Then $2|(n-1)$ and so
$$(-1)^{\frac{n(n-1)}{2}}2^{n-1}(n!)^{n-1}n!\E_n(-n)\sim -n!\E_n(-n).$$

If $n=3$, then one can compute and obtain that $-3!\E_3(-3)=-3$
which is non-square. Now let $n>3$. Then there
exists a prime number $p$ with $p\equiv 3\pmod 4$ and $p|(n-4)$.
We claim that $-n!\E_n(-n)$ is a quadratic non-residue
modulo $p$. From this claim, one derives that ${\rm Disc}_{\Q}(n!\E_n)$
is non-square. It remains to show the truth of the claim.
We divide its proof into the following two cases.

{\sc Case 1.} $p=3$. Then $3|(n-4)$, and so $n\equiv 1 \pmod 3$. Thus
\begin{align*}
-n!\E_n(-n)&=-\Big((-n)^n+2\sum\limits_{i=0}^{n-1}\frac{n!}{i!}(-n)^i\Big)\\
           &\equiv (-1)\cdot((-n)^n+2n\cdot(-n)^{n-1}) \\
           &=-n^n\equiv -1 \pmod 3.
\end{align*}
Since $-1$ is a quadratic non-residue modulo $3$,
$-n!\E_n(-n)$ is a quadratic non-residue
modulo $p$. The claim is proved in this case.

{\sc Case 2.} $p\ne 3$. Since $n\equiv 4 \pmod p$, one has
$\frac{n!}{i!}(-n)^i\equiv 0\pmod p$ for all integers $i$
with $0\le i\le n-5$ which implies that
\begin{align*}
-n!\E_n(-n)=& -\Big((-n)^n+2\sum\limits_{i=0}^{n-1}\frac{n!}{i!}(-n)^i\Big)\\
\equiv & -\Big((-n)^n+2\frac{n!}{(n-1)!}(-n)^{n-1}+2\frac{n!}{(n-2)!}(-n)^{n-2}\\
&+2\frac{n!}{(n-3)!}(-n)^{n-3}+2\frac{n!}{(n-4)!}(-n)^{n-4}\Big)\\
\equiv &-(-4^n+2\cdot4^n-6\cdot4^{n-1}+3\cdot4^{n-1}-3\cdot4^{n-2})\\
=&-2^{2n-4}\pmod p.
\end{align*}
Since $p\equiv 3 \pmod 4$, one calculates that
$$\left(\frac{-n!\E_n(-n)}{p}\right)
=\left(\frac{-2^{2n-4}}{p}\right)
=\left(\frac{-1}{p}\right)=-1.$$
Thus $-n!\E_n(-n)$ is a quadratic non-residue
modulo $p$. Hence the claim holds in this case.

(ii). Since $n$ is even, it follows that
$${\rm Disc}_{\Q}(n!\E_n)\sim (-1)^{\frac{n(n-1)}{2}}2\E_n(-n).$$
Let $p$ be a prime divisor of $n-1$ such that $v_p(n!)$ is odd.
One has
$$\E_n(-n)=\frac{(-n)^n}{n!}+2\sum\limits_{i=0}^{n-1}\frac{(-n)^i}{i!}
=\frac{-n^{n-1}}{(n-1)!}+2\sum\limits_{i=0}^{n-2}\frac{(-n)^i}{i!}.$$
Since $n$ is even and $p|(n-1)$, one has $p\ne 2$ and $p\nmid n$.
Thus $v_p\big(2\frac{(-n)^i}{i!}\big)=-v_p(i!)$ and
$$v_p\Big(\frac{-n^{n-1}}{(n-1)!}\Big)=-v_p((n-1)!)<-v_p(i!)$$
for all integers $i\in\{0,1,...,n-2\}$. Therefore by
the isosceles triangle principle (see, for instance, \cite{Koblitz})
and noticing that $p\nmid n$, we deduce that
\begin{align*}
v_p((-1)^{\frac{n(n-1)}{2}}2\E_n(-n))&=v_p(\E_n(-n))\\
&=v_p\Big(\frac{-n^{n-1}}{(n-1)!}
+2\sum\limits_{i=0}^{n-2}\frac{(-n)^i}{i!}\Big)\\
&=-v_p((n-1)!)\\
&=-v_p(n!).
\end{align*}
Thus $v_p((-1)^{\frac{n(n-1)}{2}}2\E_n(-n))$ is odd since
$v_p(n!)$ is odd. Hence ${\rm Disc}_{\Q}(n!\E_n)$ is not a
square of a rational number in this case.
This completes the proof of case 2.

(iii). Since $n\equiv 1\pmod 4$, it follows that
$${\rm Disc}_{\Q}(n!\E_n)\sim n!\E_n(-n).$$
Since $n\equiv 1\pmod 4$, we have
\begin{align}\label{eq6}
n!\E_n(-n)&=(-n)^n+2\frac{(-n)^{n-1}n!}{(n-1)!}
+2\frac{(-n)^{n-2}n!}{(n-2)!}+2\frac{(-n)^{n-3}n!}{(n-3)!}
+2\sum\limits_{i=0}^{n-4}\frac{(-n)^{i}n!}{i!}\nonumber\\
&=(n-2)^2n^{n-2}+2(n-2)K,
\end{align}
where
$$K:=\frac{1}{n-2}\sum\limits_{i=0}^{n-4}\frac{(-n)^{i}n!}{i!}$$
is a positive integer. Thus it implies that $(n-2)|n!\E_n(-n)$.
By the assumption, $n-2$ equals the product of an odd prime number
$p$ and a positive integer coprime to $p$ with $p$ satisfying (\ref{cond}).
So one may write $n=tp+2$, where
$\gcd(p, t)=1$. So $p|n!\E_n(-n)$. Thus $v_p(n!\E_n(-n))\ge 1$.

In what follows, we show that $\frac{n!\E_n(-n)}{p}\not\equiv 0\pmod p$.
It is clear that for all nonnegative integers $i$ with $i<n-2-p$, one has
$p^2|\frac{(-n)^in!}{i!}$. Also $(n-2)^2n^{n-2}\equiv 0\pmod {p^2}$ holds.
Noticing that $n\equiv 2\pmod p$, it then follows from (5.1) that
\begin{align}\label{eq7}
\frac{n!\E_n(-n)}{p} & \equiv \frac{2}{p}\sum\limits_{i=n-2-p}^{n-4}
\frac{(-n)^{i}n!}{i!} \nonumber\\
&\equiv \frac{2n(n-1)(n-2)}{p}\sum\limits_{i=n-2-p}^{n-4}
\frac{(-n)^i(n-3)!}{i!} \nonumber\\
&\equiv t\sum\limits_{i=n-2-p}^{n-4}(-1)^{i}2^{i+2}\frac{(n-3)!}{i!} \pmod p.
\end{align}

In (\ref{eq7}), we let $i=tp-p+j$, where $j$ runs over $0$ to $p-2$.
By Fermat's little theorem, one has $2^{p}\equiv 2 \pmod p$. Thus
\begin{align}\label{eq8}
2^{i+2}\equiv 2^{tp-p+j+2}\equiv 2^{t-1+j+2}=2^{t+1}2^{j}\pmod p.
\end{align}
Meanwhile, we have
\begin{align}\label{eq9}
\frac{(n-3)!}{i!}=& (n-3)(n-4)...(i+1) \nonumber\\
=& (tp-1)(tp-2)...(tp-p+j+1) \nonumber\\
\equiv &(-1)^{p-j-1}(p-j-1)! \pmod p.
\end{align}
Since $i+p-j-1=tp-1$, putting (\ref{eq8}) and (\ref{eq9}) into
(\ref{eq7}) gives us that
\begin{align}\label{eq10'}
\frac{n!\E_n(-n)}{p}\equiv & t\sum_{i=tp-p}^{tp-2}(-1)^{i+p-j-1}2^{t+1} 2^j(p-j-1)! \nonumber \\
\equiv & (-1)^{tp-1} 2^{t+1}t\sum_{j=0}^{p-2}2^j  (p-j-1)! \nonumber\\
\equiv & (-1)^{tp-1} 2^{t+1}t\sum\limits_{j=1}^{p-1}2^{p-1-j}j!\pmod p.
\end{align}
By (\ref{cond}), one has $p\nmid\sum\limits_{i=1}^{p-1}2^{p-1-i}i!$.
Since $\gcd(t,p)=1$ and $p$ is odd, it then follows from
(\ref{eq10'}) that
$$\frac{n!\E_n(-n)}{p}\not \equiv 0 \pmod p.$$
Thus $v_p(n!\E_n(-n))<2$. But we already have proved that $v_p(n!\E_n(-n))\ge 1$.
Hence $v_p(n!\E_n(-n))=1$. It infers that $n!\E_n(-n)$ is
nonsquare, and so is ${\rm Disc}_{\Q}(n!\E_n)$.
So part (iii) is proved.

This finishes the proof of Theorem \ref{thm4}. \hfill$\Box$
\bibliographystyle{amsplain}

\end{document}